\documentclass[12pt]{article}

\usepackage{tikz,amsthm,amsmath,amstext,amssymb,epsfig,euscript, mathrsfs, dsfont,pspicture,multicol,graphpap,graphics,graphicx,times,enumerate,subfig,sidecap,
wrapfig, color}
\usepackage{enumerate}

\def\ms{\medskip}

\def\1{{\mathds 1}}
\usepackage{hyperref}
\hypersetup{backref,  pdfpagemode=FullScreen} %,  colorlinks=true}
\usepackage{esint}

\def\barintgerm_#1{\mathchoice
{\mathop{\vrule width 6pt height 3 pt depth -2.5pt
\kern -8.8pt \intop}\nolimits_{#1}}%
{\mathop{\vrule width 5pt height 3 pt depth -2.6pt
\kern -6.5pt \intop}\nolimits_{#1}}%
{\mathop{\vrule width 5pt height 3 pt depth -2.6pt
\kern -6pt \intop}\nolimits_{#1}}%
{\mathop{\vrule width 5pt height 3 pt depth -2.6pt \kern -6pt
\intop}\nolimits_{#1}}}

\theoremstyle{plain}
\newtheorem{theorem}{Theorem}

\newtheorem{lemma}[theorem]{Lemma}
\newtheorem{proposition}[theorem]{Proposition}

\theoremstyle{definition}
\newtheorem{example}[theorem]{Example}

\newtheorem{remark}[theorem]{Remark}

\numberwithin{equation}{section}
\numberwithin{theorem}{section}

\begin{document}

\title{Non-relativity of K\"ahler manifold and complex  space forms
\thanks{This project was partially  supported by NSF of China (Grant No.
11301215, 11601422, 11671270, 11871044), Natural Science Foundation of Shannxi Province (2019JQ-398) and Scientific Research Program of Shaanxi Provincial Education
Department (19JK0841).
}}
\author{Xiaoliang Cheng
%\footnote{Corresponding author}
\and  Yihong Hao
}

\maketitle

\abstract

We study the
non-relativity for  two real analytic K\"ahler manifolds   and  complex space forms of three types.
The first one is a K\"ahler manifold whose polarization of  local K\"ahler potential   is  a Nash function
in a local coordinate.
The second one is  the Hartogs domain equpped with two canonical metrics whose 
 polarizations of the K\"ahler potentials are the diastatic functions.

\ms

\ms

\tableofcontents

\section{Introduction}\label{sec:1}

The problem of the existence for holomorphic isometric embeddings  from a K\"ahler manifold into complex space forms has aroused interest for
many mathematicians. In  1953, Calabi \cite{[C]} obtained an important result, i.e. the global extendability and rigidity  of a local holomorphic isometry into a complex space form. 
He provided an algebraic criterion to find out whether a complex manifold admits or not such holomorphic isometric embeddings into complex space forms.
Afterwards, there appeared many important studies about the characterization and classification of K\"ahler submanifolds
of complex space forms. Those results had been summarized in \cite{[LZ]}. 
Within the case of Hermitian symmetric spaces
of different types, Di Scala and Loi  \cite{[DL2]} generalized Calabi's
non-embeddability  result in \cite{[DL1]}. In addition, two complex manifolds are called relatives if they have a common K\"ahler submanifold with their induced metrics.

In 1987, Umehara \cite{[U]} proved that two complex space forms with different curvature signs cannot have a common K\"ahler
submanifold with their induced metrics. In 2017, Cheng, Scala and Yuan \cite{[CSY]} gave necessary and sufficient conditions for Fubini-Study space of finite dimension and different curvatures   to be relatives, which is a non-trivial generalization  of Umehara's results.

For the relativity between K\"ahler manifold and projective K\"ahler manifold,
Di Scala and Loi \cite{[DL2]} proved that a bounded domain with its Bergman metric
is not relative to  any projective K\"ahler manifold  in 2010.
They also showed that Hermitian symmetric spaces of non-compact  and any projective K\"ahler manifold are not weakly relatives.
 This implies that Hermitian symmetric spaces of compact and noncompact type are not relatives to each others.
 Loi and Mossa \cite{[LM]} showed that a   bounded homogeneous domains with a homogeneous K\"ahler metric and any projective K\"ahler manifold are not relatives in 2015. Zedda  \cite{[Z]} gave a sufficient condition for a K\"ahler manifold are strongly not relative to any projective K\"ahler manifold in 2017. As an application, they got that the Bergman-Hartogs domain and Fock-Bargmann-Hartogs domain are strongly not relative to any projective K\"ahler manifold.

For the relativity between K\"ahler manifold and the complex Euclidean
space  $\mathbb{C}^n$ with the flat metric $g_0$, Huang and Yuan \cite{[HY2]} proved that  a Hermitian symmetric space of noncompact type and $(\mathbb{C}^n, g_0)$ are not relatives by using different argument in 2015. Cheng and Niu \cite{[CN]} proved that  the Cartan-Hartogs domain equipped with it's Bergman metric
and $(\mathbb{C}^n, g_0)$ are not relatives in 2017. 
Su, Tang and Tu \cite{[STT]} proved  the symmetrized polydisc endowed with its canonical metric and $(\mathbb{C}^n, g_0)$ are not relatives in 2018.

\ms
In this paper, firstly, we study the
non-relativity for  a K\"ahler manifold   and  three types complex space forms 
$(\mathbb{CP}_{b}^{n}, g_{FS})$, $(\mathbb{CH}_{b}^{n}, g_{hyp})$ and $(\mathbb{C}^n, g_0)$ by using the technique developed
in \cite{[HY2]}.
The key condition is the polarization of any local K\"ahler potential  of this K\"ahler manifold is  a Nash algebraic function.
Our results can be seen as a generalization of \cite{[CN]} \cite{[HY2]}\cite{[STT]}.
After that, we give three criteration theorems on the non-relativity  between  the Hartogs domain 
and  complex space forms of three types by using Calabi's rigidity theorem.

\section{Nash function and the relativity}
\  \  \  \  In this section, we recall several basic facts about the complex space forms, and Nash function, which will be used in the subsequent section. Then we prove the non-existence of common K\"ahler submanifolds of the complex  space forms and a class of real analytic K\"ahler manifolds.

There are three types of complex space forms  according to the sign of their constant holomorphic sectional curvature:
\begin{itemize}
  \item Complex Euclidean space $\mathbb{C}^{n}$ is the complex linear space with the flat metric  $g_{0}$ 
  whose associated K\"ahler form 
$$\omega_{0}=\frac{\sqrt{-1}}{2}\partial \bar\partial \sum_{i=1}^{n}|z_i|^2.$$
  \item Complex projective space $\mathbb{CP}^{n}_{b}$ of complex dimension $n<\infty$, with the Fubini-Study metric
$g_{FS} $ of holomorphic bisectional curvature $4b$ for $b>0$. Let $[Z_{0},\cdots, Z_{n}]$ be homogeneous coordinates,
 $U_{0}=\{[Z_0, Z_1, \cdots, Z_n]| Z_{0}\neq 0\}.$  Define affine coordinates $z_{1},\cdots,z_{n}$ on $U_{0}$ by $z_{j}=\frac{Z_j}{b Z_{0}}, j=1, 2, \cdots, n $. 
The K\"ahler 
form is $$\omega_{FS}=\frac{\sqrt{-1}}{2b}\partial \bar\partial\log(1+b\sum_{j=1} ^{n}|z_j|^2)\ \ \text{for} \ \ b>0.$$

  \item Complex hyperbolic space $\mathbb{CH}_{b}^{n}$ of complex dimension $n<\infty$, namely the unit ball $B \subset  \mathbb{C}^{n}$ given by:
$$B=\{z\in \mathbb{C}^{n}\big| |z_{1}|^2+\cdots+|z_{n}|^2<-b\},$$
endowed with the hyperbolic metric $g_{hyp}$ of constant holomorphic sectional curvature $4b$, for $b < 0$. Fixed a coordinate system around a point $p \in B$, the hyperbolic metric $$\omega_{hyp}=\frac{\sqrt{-1}}{2b}\partial \bar\partial\log(1+b\sum_{j=1} ^{n}|z_j|^2) \ \ \text{for} \ \ b<0.$$
\end{itemize}

Let $D$ be an open subset of ${\mathbb C}^n$. Let $f$ be a holomorphic function on $D$.  $f$ is called a Nash function at $x_{0}\in D$ if there exists an
open neighbourhood $U$ of $x_{0}$ and a polynomial $ P : \mathbb{C}^n \times \mathbb{C}\rightarrow \mathbb{C}, P\neq 0$, such that $P(x, f (x)) = 0$ for $x \in U$. A holomorphic function defined on $D$ is said to
be a Nash function if it is a Nash function at every point of $D$. The family of Nash functions on $D$ we denote by $\mathcal{N}(D)$.
The examples of Nash functions on a fixed open subset of $D$ are
the restrictions of polynomials and rational functions, holomorphic on $D$. %Here, we recall
%that a function $H$ is called a holomorphic Nash algebraic function or algebraic function
%over $U \subset {\mathbb C}^\kappa$ if $H$ is holomorphic over $U$
%and there is a non-zero polynomial $P(\eta, X)$ in $(\eta, X)$ such
%that $P(\eta,H(\eta))\equiv 0$ for $\eta \in U$.  $\mathcal{N}(U)$ denotes the set of all Nash algebraic function
%on $U$.

\begin{lemma}\cite{[Tw]} \label{lemA}
If $D$ is an open subset of $\mathbb{C}^{n}$, $f\in \mathcal{N}(D)$, then
$\frac{\partial f}{\partial x_{i}} \in \mathcal{N}(D)$, for $i=1,\cdots, n.$
\end{lemma}
%See page 233 in \cite{Tw}.
\begin{lemma}\cite{[Tw]}\label{Implicit theorem}
Let $\Omega$ be an open subset of $\mathbb{C}^{n}\times\mathbb{C}^{m},$ and $(x_{0}, y_{0})\in \Omega$. Assume that 
$G: (x,y)\rightarrow G(x,y)\in \mathbb{C}^{m},$ is a Nash mapping such that $G(x_{0}, y_{0})=0$ and 
$$\det(\frac{\partial G}{\partial y}(x_{0}, y_{0})) \neq 0,$$
then exist open neighbourhoods $U$, $V$ of $x_{0}$ and $y_{0}$ respectively and  a Nash mapping $F: U\rightarrow V $such
that the Nash subset $F=G^{-1}(0)  \bigcap (U\times V)$.
\end{lemma}
\begin{lemma}\label{lemCH}
Let $f(u,v)$ be a Nash function of $u,v$ on $U \times V \subset \mathbb{C}^{n}\times\mathbb{C}^{m}$ and 
$U_1 \times V_1 \subset \mathbb{C}\times\mathbb{C}$.
Let $(L(z), H(w)):=(l_{1}(z),\cdots, l_{n}(z), h_{1}(w),\cdots, h_{m}(w))$ be a holomorphic mapping from $U_1 \times V_1 $ to $U \times V$.
Then $$\frac{\partial^{\delta} f(L(z), H(w))}{\partial w^{\delta}}\big |_{w=w_0}$$ is a Nash function of $L(z)$ for $\delta=1,2,\cdots$.
\end{lemma}

\begin{proof}
By chain rules, we know
$$\frac{\partial f(L(z),H(w))}{\partial w}
=\sum_{i=1}^{m}\frac{\partial f(u,v)}{\partial v_{i}}\Big |_{(u,v)=(L(z),L(w))}\frac{\partial h_{i}(w)}{\partial w}.$$
Since $\frac{\partial f(u,v)}{\partial v_{i}}$ is Nash function of $u, v$ by Lemma \ref{lemA}, it is true for $\delta=1$.
By mathematical induction, the result is also correct for any $\delta\in \mathbb{N}^{+}$.

\end{proof}

\begin{lemma}\cite{[HY2]}\label{lemHY}
Let $D\subset \mathbb{C}^{n} $ be a connected open set, and $\xi=(\xi_1,\cdots, \xi_k)\in D$.
Let $f_{1}(\xi),\cdots, f_{l}(\xi)$ and $f$ be Nash functions on $D$. Assume that
$$\exp f(\xi)=\prod_{i=1}^{l}f_{i}(\xi)^{\mu_i}$$
for certain real numbers $\mu_{1}, \cdots,\mu_{l}$. Then $f(\xi)$ is constant on $D$.
\end{lemma}

Let $M$ be an $n$-dimensional complex manifold endowed with a real analytic K\"ahler metric $g$. i.e. if fixed a local coordinate system $z$ on a neighbourhood $U$ of any point $p \in M$, it can be described on $U$ by a real analytic K\"ahler potential 
$\psi : U \rightarrow \mathbb{R}$. In that case the potential $\psi (z)$ can be analytically continued to an open neighbourhood 
$W \subset U \times \text{conj}(U)$ of the diagonal, where $\text{conj}(U)=\{z\in \mathbb{C}^{n}| \bar{z} \in U\}$. Denote this extension by $\psi(z,w)$. It is called the polarization of $\psi$.

\begin{theorem} \label{main theorem1}
Let $M$ be a K\"ahler manifold  admits a real analytic K\"ahler metric $g_{M}$.
Let $D \subset  \mathbb{C}$ be a connected open subset. 
Suppose that  $F: D \rightarrow \mathbb{CP}_{b}^{n} (or \ \mathbb{CH}_{b}^{n})$ and $L:D \rightarrow M $ are 
holomorphic mappings such that 
\begin{equation}\label{FL1}
F^{*}\omega_{\mathbb{C}^{n}}=L^{*}\omega_{M} \  on \ D,
\end{equation}
where the K\"ahler form $\omega_{M}=\frac{\sqrt{-1}}{2}\partial\overline{\partial} \psi$. 
%(1) If the polarization of $e^{\psi}$  is  a holomorphic Nash algebraic function,{\color{red}{------------------------------------------???????}}
If the polarization of any local K\"ahler potential $\psi$  is  a Nash function, then $F$ must be a constant map. Moreover, 
the K\"ahler manifold $(M, g_{M})$ and $(\mathbb{CP}_{b}^{n}, g_{FS}) \ (or \ (\mathbb{CH}_{b}^{n}, g_{hpy}))$  are not relatives.
\end{theorem}

\begin{proof}

%(1){\color{red}{------------------------------------------???????}}

 Assuming that $F=(f_1,
\cdots, f_n): D \rightarrow \mathbb{C}^n$ is not constant and $ L=({l_1}(z), \cdots, {l_{m}}(z)): D
\rightarrow M$ be holomorphic
maps satisfying equation \eqref{FL1}.  Without loss of generality that $D$ is simply connected, $0\in D, F(0)=0$.
By equation $\eqref{FL1}$, in local coordinate, we  have
$$\frac{1}{b}\partial\bar\partial \log (1+b\sum_{i=1}^{n}|f_i(z)|^2)=\partial\bar\partial  \psi(L(z)) ~~\text{for~~}z\in D,$$
where $\psi$ is the K\"ahler potential function of $g_{M}$.
There exists a holomorphic function $h$ on $D$ such that
\begin{equation}\label{function cp}
\log (1+b\sum_{i=1}^n |f_i(z)|^2) =b \psi(L(z))+h(z)+\overline{h}(z) \  \text{ for} \  \ z \in D.
\end{equation}
After polarization, it is equivalent to
\begin{equation}\label{polarization1}
\log(1+b\sum_{i=1}^n f_i(z) \overline{f_i}(w)) =b \psi(L(z),\overline{L}(w))+ h(z)+\overline{h}(w)~~\text{~for~~}(z, w) \in 
D\times \hbox{conj}({D}),
\end{equation}
where $\overline{f}_i(w) = \overline{f_i(\overline{w})}$, $\overline{L}(w) = \overline{L(\overline{w})}$ and
$\hbox{conj}({D})=\{z \in \mathbb{C} | \bar z \in D\}$.
\\  \

Let $w=0$, then the holomorphic function $$h(z)=-b \psi(L(z), 0)-\overline{h}(0) \ \ \text{and}\ \ 
\bar h(w)=-b \psi(0,\overline{L}(w))-h(0).$$

We divided three steps to  prove Theorem \ref{main theorem1}. \\

  \textbf{Step 1.}  For any $1\leq i \leq n$, $f_i(z)$ can be written as a  Nash function of  $L(z)$ by 
  the algebraic version of the implicit function theorem given by Lemma \ref{Implicit theorem}, shrinking $D$ toward the origin if needed.

Write
$D^\delta={\partial^\delta \over\partial w^\delta}$. 
 Applying   the
 differentiation $\partial \over\partial w$ to equation (\ref{polarization1}), we get for $w$ near $ 0$ the same form.
\begin{equation}\label{james}
\frac{\sum_{i=1}^{n} f_i(z) {\partial \over\partial w} \bar{f_i}(w)}{1+b\sum_{i=1}^n f_i(z) \bar{f_i}(w)}=\frac{\partial \psi({L}(z),\overline{{L}}(w))}{\partial w}.
\end{equation}
By Lemma \ref{lemA},  $\frac{\partial \psi({L}(z),\overline{{L}}(w))}{\partial w}\mid_{w=0}$ is a Nash function of $L(z)$.
We can rewrite it as
follows:
\begin{equation}\label{james-3}
\frac{F(z) \cdot D^1(\bar F(w)) }{1+bF(z) \cdot \bar F(w) }= \phi_1(w,  {l_1}(z), \cdots, {l_{m}}(z)),
\end{equation}
where  $\phi_1(w,{l_1}(z), \cdots, {l_{m}}(z))$ is
 a Nash  function in $L=({l_1}(z), \cdots, {l_{m}}(z))$ for fixed  $w=0$.
 
 $$\frac{bF(z) \cdot D^{2}\bar{F}(w)}{1+bF(z) \cdot \bar{F}(w)}-\frac{b^{2}(F(z) \cdot D\bar{F}(w))^{2}}{[1+bF(z) \cdot \bar{F}(w)]^{2}}=\phi_{2}(w,  {l_1}(z), \cdots, {l_{m}}(z)).$$
Differentiating the equation above, for the fixed point $w=0$, we get for any
$\delta$ the following equation 

\begin{equation}
\label{implicit} bF(z) \cdot D^{\delta}(\bar F(0))+P_{\delta}(F(z))
=\phi_{\delta}( {l_1}(z), \cdots, {l_{m}}(z)),
\end{equation}
where $P_{\delta}(F(z))$ is polynomial in $F(z)$ for any $\delta$ and fixed $z$ and has no constant and linear terms in the Taypor expansion with respect to $F$.
Here for $\delta> 0$ and the fixed $w=0$, $\phi_{\delta}(w,  {l_1}(z), \cdots, {l_{m}}(z))$
is a Nash function  in
${l_1}(z), \cdots, {l_{m}}(z)$.

Let ${\mathcal
L}:=\hbox{Span}_{\mathbb C}\{ D^{\delta}(\bar F(w))|_{w=0}\}_{\delta\ge
1}$ be a vector subspace of ${\mathbb C}^{n}$. Let
$\{D^{\delta_j}(\bar F(w))|_{w=0}\}_{j=1}^{\tau}$ be a basis for
$\mathcal L$. Then for a small open disc $\Delta_0$ centered at $0$ in
${\mathbb C}$, $\bar F(\Delta_0)\subset {\mathcal L}.$ Indeed, for any
$w$ near $0$, we have 
$$\bar F(w)=\bar F(0)+\sum_{ \delta \ge 1}{D^{\delta}(\bar F)(0)\over\delta
!}w^{\delta}=\sum_{ \delta \ge 1}{D^{\delta}(\bar F)(0)\over\delta
!}w^{\delta}\in {\mathcal L}.$$
from the Taylor expansion.

Let $\nu_j%=(\nu_{j1},\cdots, \nu_{jN})
~(j=1 \cdots, n-\tau)$ be a basis of the Euclidean orthogonal
complement of ${\mathcal L}.$  Then, we have

\begin{equation}
\label{implicit t} F(z)\cdot \nu_{j} =0, ~~ \text{for~each}~~ j=1,\cdots,
n-\tau.
\end{equation}
Consider the system consisting of (\ref{implicit}) at $w=0$ (with
$\delta=\delta_1,\cdots,\delta_\tau$) and (\ref{implicit t}).
The linear coefficient matrix in the left hand side of
the system at $w=0$ with respect to $F(z)$ is
\begin{equation}\notag
\begin{bmatrix}
D^{\delta_1}(\bar F(w))|_{w=0}\\
\vdots\\
D^{\delta_\tau}(\bar F(w))|_{w=0}\\
\nu_1\\
\vdots\\
\nu_{n-\tau}
\end{bmatrix}
\end{equation}
and is obviously invertible. Note that the right hand side of the
system of equations consisting of (\ref{implicit}) at $w=0$ (with
$\delta=\delta_1,\cdots,\delta_\tau$)
is a  Nash function in
$L$. By  the algebraic version of the implicit function theorem given by Lemma \ref{Implicit theorem}, there exist Nash algebraic functions $\hat f_i(z, X_1, \cdots, X_m)$ in $\hat U$ such that $f_i(z) = \hat f_i(z, l_1(z), \cdots, l_m(z))$ for  $z \in U$.

  \textbf{Step 2.} If all of the elements  $l_{1}, \cdots,  l_{m}$ are  Nash
functions, we know $f_{i}(z)$ is a holomorphic Nash algebraic function by Step 1.
We consider the equation $\eqref{polarization1}$ by the method in $[11]$. The equation is equivalent to the following
\begin{eqnarray*}
 1+b\sum_{i=1}^n f_i(z) \bar{f_i}(w)
&=&e^{ b\psi(L(z),\overline{L}(w))} e^{-b\psi(L(0),\overline{L}(w))-b\psi(L(z),\overline{L}(0)) -h(0)-\overline{h}(0))}.
\end{eqnarray*}

By Lemma \ref{lemHY}, $F$ is a constant map.

  \textbf{Step 3.}
Suppose there exist some elements in $(l_{1}, \cdots, l_{m})$, which are not Nash algebraic functions.
Let $\mathfrak{R}$ be the field of  Nash algebraic functions in $z$ over $D$. Consider the field extension 
$$\mathfrak{F}=\mathfrak{R}(l_{1}, \cdots, l_{m}),$$
 namely, the smallest subfield of meromorphic function field over $D$ containing  Nash algebraic functions and 
 $l_{1}, \cdots, l_{m}$.
 Without loss of generality, 
let $\mathcal{L}= \{l_1(z), \cdots, l_r(z)\}$ be the maximal algebraic independent subset in $\mathfrak{F}$, thus the transcendence degree of $\mathfrak{F} / \mathfrak{R}(\mathcal{G})$ is 0. Then there exists a small connected open subset $U$ with $0\in\overline{U}$ such that for each $j$ with $l_{j}\not \in \mathcal{L}$, we have some  Nash  functions 
$\{\hat l_{j}(z, X)\}$ in the neighborhood
  $\hat U$ of $\{ (z, l_1(z), \cdots, l_r(z)) | z\in U \}$ in $\mathbb{C}\times\mathbb{C}^r$ such that it holds that for any
    $l_{j} \not\in \{l_1, \cdots, l_r\}$,
    
    $$l_{j}(z) = \hat l_{j} (z, l_1(z), \cdots, l_r(z))$$ for any $z \in U$, where $X=(X_1, \cdots, X_r)$.
    
By the first step, we have seen that for fixed $w=0$, $\frac{\partial \log \psi(L(z),\overline{L}(w)}{\partial w}$ also is a Nash function in
$(l_{1}, \cdots, l_{r}) $. So there exists  a  Nash algebraic function $\hat f_i(z, X)$ in $\hat U$ such that $$f_i(z) = \hat f_i(z,l_1(z), \cdots, l_r(z)), i=1,\cdots, n$$ for  $z \in U$.

Denote $\hat{L}(z,X)=(\hat l_{r+1}(z, X), \hat l_{r+2}(z, X), \cdots, \hat l_{m}(z, X))$. Define a function 
as follows:
\begin{eqnarray*}
\Psi(z, X, w)&=&\log(1+b\sum_{i=1}^n \hat f_i(z, X) \bar{f_i}(w)) -b \psi((X,\hat{L}(z,X)),\overline{L}(w))\\
&&+b\psi((X,\hat{L}(z,X)), 0)+\overline{h}(0)-\overline{h}(w)~~\text{~for~~}(z, w) \in \hat{U}\times U.
\end{eqnarray*}
Then $\Psi(z, l_1, \cdots, l_r, w)\equiv 0$ on $U$. 
We claim that $\Psi(z, X, w)\equiv 0$ on $\hat{U}\times U$.
In fact, it only need to prove that $\frac{\partial \Psi}{\partial w}(z, X, w)\equiv 0$ on $\hat{U}\times U$.

Otherwise, then there exists a neighborhood $U_{0}$ of $0 \in U$ such that  $\frac{\partial \Psi}{\partial w}(z, X, w_{0})\neq 0$. For fixed
$w_{0}\in U_{0} , \frac{\partial \Psi}{\partial w}(z, X, w_{0})$ is a Nash function in $(z, X)$. Assume that its annihilating function is
$P(z,X,t) = a_{d}(z,X)t^{d} + \cdots + a_{0}(z,X),$
where $a_{0}(z, X)\neq 0$ on $U $, and ${a_{i}(z, X)}$ are holomorphic polynomials in $(z, X)$. Note that
$\Psi(z, l_1, ..., l_m,w_0)=0$  on $V$. Then $\frac{\partial \Psi}{\partial w}(z, l_1, \cdots, l_r,w_0)=0$ on $V$. Hence,
 \begin{eqnarray}
P(z, l_1(z), \cdots, l_r(z), \frac{\partial \Psi}{\partial w}(z, l_1(z), \cdots, l_r(z), w_0) )
   & = &P(z, l_1(z), \cdots, l_r(z), 0)\\
   & = &a_{0}(z, l_1, \cdots,  l_{r})=0.
 \end{eqnarray}
Therefore $\{l_1(z), \cdots, l_r(z)$\} are algebraic dependent over $\mathfrak{R}$. This is a contradiction.
 
 We have the following equation:
 \begin{eqnarray}
 \log(1+b\sum_{i=1}^n \hat f_i(z, X) \bar{f_i}(w))
  & = & b \psi((X,\hat{L}(z,X)),\overline{L}(w))-b\psi((X,\hat{L}(z,X)),0)\\
&&-\overline{h}(0)+\overline{h}(w)~~\text{~for~~}(z, w) \in \hat{U}\times U.
\end{eqnarray}

 If we have the equation
\begin{equation}\notag
\sum_{i=1}^n \hat f_i(z, X) \bar f_i(w) = 0,
\end{equation} then $\displaystyle\sum_{i=1}^{n} |f_i(z)|^2=\sum_{i=1}^{n} \hat f_i(z, l_1, \cdots, l_r) \bar f_i(z)=0$ by taking $w=z$.
This implies that $F$ is a constant map, which contradicts with the previous assumption.

If  there exist  $z_0, w_0$ such that
$$\displaystyle\sum_{i=1}^n \hat f_i(z_0, X) \bar f_i(w_0)\neq 0,$$  then 
 $\displaystyle\sum_{i=1}^n \hat f_i(z_{0}, X) \bar f_i(w_0)$ is a nonconstant Nash function in $X$.
%If not. we taking  $w=\overline{z}$ and $X=(g_1(z), \cdots, g_l(z))$, %in equation (\ref{zero}), the right hand side is identity 1. I
%it follows that
%$$\sum_{i=1}^n |f_i(z)|^2 = \sum_{i=1}^n \hat f_i(z, g_1(z), \cdots, g_l(z)) \bar f_i(\overline{z}) \equiv 0,\ \hbox{over}\ U.$$ This implies that $f_i(z) \equiv 0$ for all $1 \leq i \leq n$
%.However,we have supposed that $F=(f_1, \cdots, f_n)$ is a non-constant map,it is a contradiction.
%\end{proof}
 Consider the following equation  %as in Lemma  \ref{nonconstant}, %for equation (\ref{zero}), %one has:
$$1+b\sum_{i=1}^n \hat f_i(z_{0}, X) \bar{f_i}(w_0)
 =e^{-\overline{h}(0)+\overline{h}(w_0)} e^{b \psi((X,\hat{L}(z_{0},X)),\overline{L}(w_0))-b\psi((X,\hat{L}(z_{0},X)),0)}.$$

%$$\sum_{i=1}^d \hat f_i(z_0, X) \bar f_i(w_0) = \log \psi(X, \hat g_{1}(z_0, X), \cdots, \hat {g}_{l}(z_0, X), \overline{l_{1}}(w_0),\cdots,\overline{l_{d}}(w_0)).$$

%$$\exp ({\sum_{i=1}^d \hat f_i(z, X) \bar{f_i}(w)})=e^{-(h(0)+\overline{h}(0))}\frac{\psi(X,G(z),\overline{L}(w))}{\psi(L(0),\overline{L}(w))\psi(X, G(z),\overline{L}(0))}.$$

The right hand is  nonconstant holomorphic  Nash algebraic functions in $X$.
 It follows that $F$ is constant from Lemma \ref{lemHY}.
% that the left hand side also is growth polynomially, and the right hand side is  logarithmic growth when approaching poles. This is a contradiction.
 The proof of Theorem 1.1 is completed.%\ref{main1}

\end{proof}
%\begin{theorem} 
%Let $M$ be a K\"ahler manifold  admits a real analytic K\"ahler metric $g_{M}$.
%Let $D \subset  \mathbb{C}$ be a connected open subset. 
%Suppose that  $F: D \rightarrow \mathbb{CH}^{n}$ and $L:D \rightarrow M $ are 
%holomorphic mappings such that 
%\begin{equation}\label{FL}
%F^{*}\omega_{\mathbb{CH}^{n}}=L^{*}\omega_{M} \  on \ D,
%\end{equation}
%where the K\"ahler form $\omega_{M}=\frac{i}{2}\partial\overline{\partial} \psi$. 
%(1) If the polarization of $e^{\psi}$  is  a holomorphic Nash algebraic function,{\color{red}{------------------------------------------???????}}
%If the polarization of any local K\"ahler potential $\psi$  is  a Nash algebraic function, then $F$ must be a constant map. Moreover, $(M, g_{M})$ and $(\mathbb{CP}^{n}(b), g_{\mathbb{C}^{n}})$  are not relatived.
%\end{theorem}

\begin{example}
Let $P(z,w)$ be a holomorphic polynomial over $\mathbb{C}^{2n}$ such that $g_{P}=\frac{i}{2}\partial \bar\partial P(z,\bar{z})$ 
 is the K\"ahler form associated to some K\"ahler metric $g_{P}$. Then
$(\mathbb{C}^{n}, g_{P})$ is not relatives with $\mathbb{CP}^{n}_{b}$ and $\mathbb{CH}^{n}_{b}$.
This result can be seen as a generalization of  Umehara's result in \cite{[U]}.

\end{example}

 \begin{theorem} \label{main theorem2}
Let $M$ be a K\"ahler manifold admits a real analytic K\"ahler metric $g_{M}$.
Let $D \subset  \mathbb{C}$ be a connected open subset. 
Suppose that    $F: D \rightarrow \mathbb{C}^{n}$ and $L:D \rightarrow M $ are 
holomorphic mappings such that 
\begin{equation}\label{FL}
F^{*}\omega_{\mathbb{C}^{n}}=L^{*}\omega_{M} \  on \ D,
\end{equation}
where the K\"ahler form $\omega_{M}=\frac{\sqrt{-1}}{2}\partial\overline{\partial} \psi$. 
If the polarization of $e^{ \psi}$  is  a holomorphic Nash algebraic function,
then $F$ must be a constant map. Moreover, $(M, g_{M})$ and $(\mathbb{C}^{n}, g_{\mathbb{C}^{n}})$  are not relatives.

%(2) If the polarization of $\psi$  is  a holomorphic Nash algebraic function, 
%then 
%{\color{red}{------------------------------------------???????}}

%f(z) is nash function satisfies the function. $fj(z)=$
% $$\sum_{i=1}^d f_i(z) \bar{f_i}(0) =\psi(L(z),L(0))+h(z)+\overline{h}(0) \  \text{ for} \  \ z \in D.$$

\end{theorem}

\begin{proof}
 Assuming that $F=(f_1,
\cdots, f_n): D \rightarrow \mathbb{C}^n$ is not constant and $ L=({l_1}(z), \cdots, {l_{m}}(z)): D
\rightarrow M$ be a holomorphic
map satisfying equation \eqref{FL}.  Without loss of generality that $D$ is simply connected, $0\in D, F(0)=0$.
By equation $\eqref{FL}$, we  have
$$\partial\bar\partial(\sum_{i=0}^{d}|f_i(z)|^2)=\partial\bar\partial \log [e^{\psi(L(z))}] ~~\text{for~~}z\in D,$$
where $\psi$ is the global K\"ahler potential function of $g_{M}$ on $M$.
There exists a holomorphic function $h$ on $D$ such that
\begin{equation}\label{function}
\sum_{i=1}^d |f_i(z)|^2 =\log [e^{\psi(L(z))}]+h(z)+\overline{h}(z) \  \text{ for} \  \ z \in D.
\end{equation}
After polarization, (\ref{function}) is equivalent to
\begin{equation}\label{polarization}
\sum_{i=1}^d f_i(z) \bar{f_i}(w) =\log [e^{ \psi(L(z),\overline{L}(w)}]+ h(z)+\overline{h}(w)~~\text{~for~~}(z, w) \in D\times
\hbox{conj}({D}),
\end{equation}
where $\bar{f}_i(w) = \overline{f_i(\overline{w})}$, $\bar{L}(w) = \overline{L(\overline{w})}$ and
$\hbox{conj}({D})=\{z \in \mathbb{C} | \bar z \in D\}$.
Let $w=0$, then the holomorphic function $$h(z)=-\log [e^{\psi(L(z),\overline{L}(0))}]-\overline{h}(0).$$

By the similar argument in the proof of Theorem \ref{main theorem1}, 
we  can prove Theorem \ref{main theorem2}. Hence we omit here.

\end{proof}

\begin{remark}
The results can be generalized to indefinite complex space forms:
$(\mathbb{C}^{n,s}, \omega_{\mathbb{C}^{n}}) $,
$(\mathbb{CP}^{n,s}, \omega_{\mathbb{CP}^{n}}) $ and 
$(\mathbb{CH}^{n,s}, \omega_{\mathbb{CP}^{n}}) $,
see the definition in \cite{[CSY]}.
\end{remark}

%\section{Further Remark}
%%Let $\Omega_j(j=1,2,\cdots,m) $ are all the Classic domains,  the corresponding Cartan-Hartgos domains are ${M_{\Omega_j}(\mu_j)}$  with the Bergman metric 
%$\omega_{M_{\Omega_j}}$. Indeed, we can  prove the following general result:

%\begin{theorem} %\label{main1}
%Let $D \subset \CC$ be a connected open subset. Suppose that
%$F: D \rightarrow \CC^n$ and $L=(G_1, \cdots, G_m): D \rightarrow
%M_\Omega=M_{\Omega_1} \times \cdots \times M_{\Omega_m}$ are holomorphic
%mappings such that
%\begin{equation}\label{isometry}
%F^*\omega_{\CC^n} =\sum_{j=1}^m \mu_j G_j^*\omega_{M_{\Omega_j}} ~~\text{on}~~D
%\end{equation}
% for certain real  constants $\mu_1,
%\cdots, \mu_m$. Then $F$ must be a constant map. Furthermore, if all
%$\mu_j's$ are positive, then  $G$ is also a constant map.
%\end{theorem}

\section{The relativity between Hartogs domain and complex space forms}

Let $D\subset \mathbb{C}^{d}$ be a domain and $\varphi$ be a continuous positive function on $D$.
The domain
\begin{equation}\label{equ:Hartogs domain}
\Omega=\left\{(\xi, z)\in \mathbb{C}^{d_{0}}\times D : ||\xi||^{2}<\varphi(z)\right\}
\end{equation}
is called a Hartogs domain over  $D$ with $d_{0}$-dimensional fibers.
In this section, we consider the Hartogs domain equiped with its Bergman metric or almost K\"ahler-Einstein metric.
%\begin{itemize}
%  \item[1.]$\Omega$ is pseudoconvex
%if and only if $D$ is pseudoconvex and $-\log \varphi$ is plurisubharmonic. 
 % \item[2.]  $(\varphi-||\xi||^{2})^{-1}$  is a Nash function on $\Omega$ form if $\varphi$ is Nash function on $D$.

%  \item[3.] $\Omega$ is circular with the origin if and only if
% $D$  is   circular with the origin and $\varphi(e^{\sqrt{-1}\theta}z)=\varphi(z)$.
%\end{itemize}

\subsection{Bergman metric}
In \cite{[Li]}, Ligocka gave a series representation formula of the
Bergman kernel of the Hartogs domain involving weighted Bergman kernels of the base domain.
She proved  that the Bergman kernel of $\Omega$
\begin{equation}\label{equ:Ligocka formula2}
K_{\Omega}(z,\xi),(w,\zeta)=\sum_{k=0}^{\infty}\frac{(k+1)_{d_{0}}}{\pi^{d_{0}}}K_{D,\varphi^{k+d_{0}}}(z,w)<\xi,\zeta>^{k},
\end{equation}
where $K_{D,\varphi^{k+d_{0}}}$ stands for the weighted Bergman kernel  with respect to the weight $\varphi^{k+d_{0}}$,
$<\cdot , \cdot> $ denotes the scalar product in $\mathbb{C}^{d_{0}}$,
$(k+1)_{d_{0}}$ denotes the Pochhammer polynomial of degree $d_{0}$, i.e.
$(k+1)_{d_{0}}=\frac{\Gamma(k+1+d_{0})}{\Gamma (k+1)}.$
When  $D$ is a bounded homogeneous domain  and $\varphi(z)=K(z,z)^{-s}$,
Ishi, Park and Yamamori \cite{Park} proved that the Bergman kernel is 
\begin{equation}\label{B kernel}
K_{\Omega}((z,\xi),(w,\eta))=\frac{K_{D}(z,w)^{d_{0}s+1}}{\pi^{d_{0}}}\sum_{j=0}^{d}\frac{c(s,j)(j+d_{0})!}{(1-t)^{j+d_{0}+1}}|_{K
_{D}(z,w)^{s}<\xi,\eta>},
\end{equation}
where the constants $c(s,j)$ satisfy that 
$F(ks)=\sum_{j=0}^{d} c(s, j)(k+1)_{j}$
and $F$ is the polynomial given by (18) in \cite{Park}.

 It is well known that any bounded homogeneous domain are equivalent to 
a homogeneous Siegel domain of the second kind. 
Recall the definition of  homogeneous Siegel domain of the second kind (see p.10 Def 1.2-1.3 in \cite{Gi}):
$$D(V,F)=\{(z,u)\in \mathbb{C}^{n}\times\mathbb{C}^{m}|Im(z)-F(u,u)\in V\}.$$
where $V$ is a convex cone in $\mathbb{R}^{n}$ not containing any straight lines and $F(u,v):\mathbb{C}^{n}\times\mathbb{C}^{m} \rightarrow  \mathbb{C}^{n}$ is a $V$-Hermitian form.
%\begin{itemize}
%\item $F(\lambda_{1}u_{1}+(\lambda_{2}u_{2},v)=\lambda_{1}F(u_{1}+v)
%+\lambda_{2}F(u_{2},v), \lambda_{1},  \lambda_{2} \in \mathbb{C}, $
%\item $F(u,v)=\overline{F(v,u)},$
%\item $F(u,u)\in \overline{V}$, where $\overline{V}$ is the closure of the cone V.
%\item $F(u,u)=0$ only if $u=0.$
%\end{itemize}
The Bergman kernel  of $D(V,F)$  is $$B(\eta,\zeta)=c\Big(\frac{w-\bar{z}}{2i}-F(v,u)\Big)^{2d-q},$$
where $c$ is a constant. See Theorem 5.1 in \cite{Gi}. 
Since $B(\eta,\zeta)$ is a Nash function,  $(D(V,F), g_{B})$ is not relatived with  $\mathbb{C}^{n}$  by Theorem  \ref{main theorem2}.
The property that the Bergman metrics are isometric between two biholomorphic equivalent domains implies the following result.

\begin{proposition}\label{BHD}
Any bounded homogeneous domain equipped with its Bergman metric and $\mathbb{C}^{n}$ are not relatives.
\end{proposition}

By Theorem \ref{main theorem2} and the formula \eqref{B kernel}, we have the following theorem.

\begin{theorem}
Let $\Omega$ be a Hartogs domain  over bounded homogeneous domain $D$, $\phi(z)=K(z,z)^{-s}, $
 then $(\Omega, g_{B})$ and $\mathbb{C}^{n}$ are not  relatives.
\end{theorem}

\subsection{Almost K\"ahler-Einstein metric}
Let $g_{\Omega}$ be an almost K\"ahler-Einstein metric  given by  the boundary, i.e.
$$\omega_{\Omega}=-\frac{\sqrt{-1}}{2} \partial \overline{\partial}\log (\varphi-||\xi||).$$
Obviously, the polarization of $(\varphi(z)-||\xi||^{2})^{-1}$  is a Nash function $\Omega\times \Omega$ if and only if the polarization of  $\varphi$ is a Nash function on $D\times D$.
By   Theorem \ref{main theorem2}, we have the following result.
\begin{theorem}
Let $\Omega$ be a Hartogs domain given by \eqref{equ:Hartogs domain}, if 
the polarization of  $\varphi$ is a Nash function on $D\times D$. Then  $(\Omega, g_{\Omega})$ and $\mathbb{C}^{n}$ are not relatives.
\end{theorem}

In general, the polarization of  $\varphi$ may not be a Nash function. In this situation, we also can consider the relative problem between Hartogs domain and complex space forms by using Calabi's rigidity theorem.

%\begin{theorem}Calabi rigidity theorem
%\end{theorem}
The following lemma shows the  relation  between the diastatic functions of   $g_{D}$ and $g_{\Omega}$.
The diastatic function was introduced by Calabi in \cite{[C]}. It is a special (local) potential function. For more details, see
\cite{[C]} and \cite{[LZ]}.

\begin{lemma}\cite{HW}\label{lem:diastasis} 
 Let $g_{D}$ be a K\"ahler metric given by   $\frac{\sqrt{-1}}{2} \partial \overline{\partial} (-\log \varphi)$.    If $-\log\varphi$ is the diastasic function centered  at the origin for $(D, g^{D})$, then
 $-h\log(\varphi(z)-||\xi||^{2})$  is  the diastasic function centered  at the origin for $hg^{\Omega} $ for $h>0$.
\end{lemma}

The existence of  the full K\"ahler immersion from
  the Hartogs domain to $\mathbb{C}^{N}$, $\mathbb{CP}^{N}$ and $\mathbb{CH}^{N}$
  for  $N\leqslant +\infty$  has been discussed by Wang and Hao in \cite{HW}. 
  
  \begin{lemma}\cite{HW}\label{psd cp}
  \label{thm:CPS}
Let $\Omega$ be  as in \eqref{equ:Hartogs domain} and  $h$ be a positive number. Suppose that $\Omega$ is
a simply connected circular domain with center zero and
the function $-\log\varphi$ is the special K\"ahler potential function of  $g^{D}$ determined by its diastatic function.
Then $(\Omega, \alpha g^{\Omega})$ admits a full K\"ahler immersion into  $(\mathbb{CP}^{\infty}, g_{FS})$
if and only if $(D, (\alpha+\sigma)g^{D})$  admits a K\"ahler immersion into
$(\mathbb{CP}^{n}, g_{FS})$ or $(\mathbb{CP}^{\infty}, g_{FS})$ for all $\sigma\in \mathbb{N}$.

  \end{lemma}
  
 The formula of the full K\"ahler immersion can be expressed by the K\"ahler immerson from the base space to
 complex space forms. 
 
 \begin{lemma}\label{lem f}
 If $f: (\Omega,g_{\Omega})\rightarrow \mathbb{CP}^{\infty}$ is a holomorphic map such that
 $f^{*}\omega_{FS}=\alpha \omega_{\Omega}$, then up to an unitary transformation of $\mathbb{CP}^{\infty}$,  it is
 given by : 
 \begin{equation}\label{f}
 f=[1,s,h_{\alpha},\cdots,\sqrt{\frac{(m+\alpha-1)}{(\alpha-1)!m!}}h_{\alpha+m} w^{m},\cdots]=[1,s,H],
 \end{equation}
 where $s=(s_{1},\cdots,s_{m},\cdots)$ with:
 $$s_{m}=\sqrt{\frac{(m+\alpha-1)}{(\alpha-1)!m!}} w^{m},$$
 and $h_{\beta}=(h_{\beta}^{1},\cdots,h_{\beta}^{j},\cdots)$ denotes the sequence of holomorphic maps on
 $D$ such that the immersion $\tilde{h}_{\beta}=(1,h_{\beta}) : (D, \omega_{D})\rightarrow\mathbb{CP}^{\infty}$, satisfies
$ h^{*}_\beta \omega_{FS}=\beta\omega_{D}$, i.e.:
 \begin{equation}\label{hk}
 1+\sum_{j=1}^{\infty}|h_{\beta}^{j}|^{2}=\frac{1}{\varphi^{\beta}}.
 \end{equation}
 \end{lemma}
 
 \begin{proof}
 Since the K\"ahler immersion $f$ is isometric, by Proposition 6 in \cite{[C]}  about distasis function, we have
 $$\frac{1}{(\varphi-|w|^2)^{\alpha}}=\sum_{j=0}^{\infty}|f_{j}|^{2},$$
 for $f=[f_{0},\cdots,f_{j},\cdots]$. 
 %If we consider the power expansion around the 
 %origin of the left hans side with respect to w, w, we get
 %$$\sum_{k=1}^{\infty}
 %(\frac{\partial ^{2k}}{\partial w^{k} \partial \bar{w}^{k}}\frac{1}{\varphi-|w|^{2k}}\frac{|w|^{2k}}{k!^{2}})
 %=
 %(\frac{\partial ^{2k}}{\partial w^{k} \partial w^{k}}\frac{1}{1-|w|^{2k}}\frac{|w|^{2k}}{k!^{2}})
 %=\frac{1}{1-|w|^{2k}}-1
 %$$
 %The power expansion with respect to z and z reads:
 %$$\sum_{j,k}
% (\frac{\partial ^{|m_{j}+|m_k||}}{\partial z^{m_{j}} \partial z^{m_{k}}}\frac{1}{\psi-|w|^{2k}}\frac{z^{m_j}+z^{m_k}}{m_j !
%  m_k !})
% =\sum_{j=1}^{\infty} |h|^{2}
% $$
%where the last equality holds since by  is the power expansion of  $\psi-1$.
By using equation \eqref{hk} and Bocher coordinates, the power expansion  is
\begin{eqnarray}
&&\sum_{m=1}^{\infty}\sum_{j,k}
 \left(\frac{\partial ^{|m_{j}+|m_k|}}{\partial z^{m_{j}} \partial \bar{z}^{m_{k}}}
 \frac{\partial ^{2m}}{\partial w^{m} \partial \bar{w}^{m}}\frac{1}{(\varphi-|w|^{2})^{\alpha}}\right)|_{0}\frac{z^{m_j}\bar{z}^{m_k}w^{m}\bar{w}^{m}}{m_j !
  m_k !m!^{2}}\\
&=&
\sum_{m=1}^{\infty}\sum_{j,k}
\left(\frac{\partial ^{|m_{j}+|m_k|}}{\partial z^{m_{j}} \partial \bar{z}^{m_{k}}}
 \frac{(m+\alpha-1)!m!}{(\alpha-1)!\varphi^{(\alpha+m)}}\right)|_{0}
 \frac{z^{m_j}\bar{z}^{m_k}|w|^{2m}}{m_j ! m_k !m!^{2}}\\
 & =&\sum_{m=1}^{\infty} \sum_{j=1}^{\infty}  \frac{(m+\alpha-1)!}{(\alpha-1)!m!}|w|^{2m}
 |h_{(\alpha+m)}^{j}|^2.
 \end{eqnarray}

It follows by the previous power series expansions, that 
the map $f$ given by \eqref{f}  is a K\"ahler immersion of $\Omega\rightarrow \mathbb{CP}^{\infty}$.
By Calabi's rigidity theorem (i.e. Theorem 2 in \cite{[C]}), all the other K\"ahler immersions are given by $U\circ f$, where $U$ is a unitary transformation of $\mathbb{CP}^{\infty}$.
 \end{proof}
 \begin{lemma}\cite{[Z]}\label{Zedda}
Let $ (M, g)$ be a K\"ahler manifold. If  $(M, ag)$  is full K\"ahler immersion submanifold of $\mathbb{CP}^{n}$ for any $a >a_{0} >0$ and if $(M,g)$ and $\mathbb{CP}^{n}$  are not relatives for any $n<+\infty$, then 
$(M,g)$ and $\mathbb{CP}^{n}$  are not  strongly relatives for any $n<+\infty$. 
 \end{lemma}
 
 \begin{theorem}\label{phd cp1}
 Let $\Omega$ be  as in \eqref{equ:Hartogs domain} and  $h$ be a positive number. Suppose that $\Omega$ is
a simply connected circular domain with center zero and
the polarization of  $-\log\varphi$ is the diastatic function of  $g^{D}$.
If $(D, (h+\sigma)g^{D})$  admits a K\"ahler immersion into
$(\mathbb{CP}^{n}, g_{FS})$ or $(\mathbb{CP}^{\infty}, g_{FS})$ for all $\sigma\in \mathbb{N}$ and $(D, g^{D})$ is strongly not relative to any projective manifold,
 then it is  strongly not relative to any projective manifold.
 \end{theorem}

 \begin{proof} 
 We prove that the pseudoconvex-Hartogs domain  is not relative to any projective manifold firstly.
Assume that $S$ is a $1$-dimensional common K\"ahler
submanifold of $\mathbb{CP}^n$ and $(\Omega,g_{\Omega})$. Then around each point $p \in S$ there
exists an open neighbourhood $U$ and two holomorphic maps $\Phi : U \rightarrow \mathbb{CP}^n$ and
$\Psi : U \rightarrow \Omega$,  $\Psi(\xi)=(\Psi_{0}(\xi),\Psi_{1}(\xi),\cdots,\Psi_{d}(\xi))=(w,z)$ , where $\xi$ are coordinates on $U$, such that
$\Phi^{*}\omega_{FS}|_{U}=\Psi^{*}(\omega_{\Omega})|_{U}$. 
   Let $f  = [f_{0},...,f_{j},...] : \Omega \rightarrow \mathbb{CP}^{\infty}$ be the full K\"ahler map given by Lemma \ref{lem f} from $(\Omega, g) $ into $\mathbb{CP}^{\infty}$.
   Then we have $$\mathbb{CP}^{n}\xleftarrow[]{\Phi} U\xrightarrow[] {\Psi}\Omega \xrightarrow[] {f} \mathbb{CP}^{\infty}.$$
We claim that   $f\circ \Psi:U\rightarrow \mathbb{CP}^{\infty}=[s(\Psi_{0}), H(\Psi_{1},\cdots,\Psi_{d})]$ is full. 
   Actually, we only need to prove that $\{s_{m}(\Psi_{0})\}_{m=1}^{\infty}$ are linearly independent.
The formula of $f$ given by Lemma \ref{lem f} implies that
  $\{s_{m}\}$   are linearly independent  subsequence of  $ \{f_{j}\}$. 
% The formula of $f$  implies that $f|_{\bigtriangleup}:\bigtriangleup\rightarrow \mathbb{CP}^{\infty}$ is full K\"ahler immersion. 

Let $q$ be any positive integer and assume that  there exist $q $ complex numbers  $a_{0},\cdots,a_{q}$ such
that 
\begin{equation}\label{equ}
a_{0}s_{0}(\Phi_{0}(\xi))+\cdots+a_{q}s_{q}(\Phi_{0}(\xi))=0, \xi \in U.
\end{equation}

It is worth to point out that $\frac{\partial \Psi_{0}}{\partial \xi}\neq 0$. Otherwise, $\Phi (U) \subset \Omega|_{w=0}=D$ ,
$U$ is a common submanifolds of $(D, a g_{B}) $ and $\mathbb{CP}^{n}$ which is contradict with Loi's result.
By the assumption on $\Phi_{0}: U\rightarrow \mathbb{C}$, it follows that $\Phi_{0}(U)$  is an open subset of $\mathbb{C}$.
 Therefore, equality \eqref{equ} is satisfied on all $\mathbb{C}$. Since $s_{1},\cdots, s_{q}$ 
   are linearly independent, so $a_{j}=0$. Therefore, $\{s_{m}(\Psi_{0})\}_{m=1}^{\infty}$ are linearly independent which implies  $$(f\circ\Phi)(\xi)=[s(\Phi_{1}), H(\Phi_{1}(\xi),\cdots,\Phi_{d}(\xi))]$$
   is non-degenerate. By  Loi's Lemma 2.2  in \cite{[DL2]}, $f\circ \Psi:U\rightarrow \mathbb{CP}^{\infty}$ is full.   
   On the other hand, $\Phi : U \rightarrow \mathbb{CP}^n$ is a K\"ahler immersion.
 It is contradict to Calabi's rigidity theorem  in \cite{[C]}.
   
Finially,  combine  Lemma  \ref{psd cp} with Lemma \ref{Zedda},
  pseudoconvex-Hartogs domain is strongly not relative to $\mathbb{CP}^{n}$ for any $n$ and   any projective manifold.

 \end{proof}
 
 %\begin{remark}If , then 
% . Hence, $ \Phi(U) \subset \Omega|_{z=0}=\bigtriangleup$
 % and Then there are a
%a subsequence $f_{j1} , . . . , f_{jm},\cdots$  of ${f_{j}}$ which restricted to $\bigtriangleup$ are linearly independent. Since
% $\Psi1(\xi)$  is not constant, so $f_{j_{1}}(\Psi_{1}(\xi)),...,f_{j_{m}}(\Psi_{1}(\xi)), \cdots,$ is
%of linearly independent functions.  It is a subsequence of $f|_{\bigtriangleup}\circ \Psi$.
%\end{remark}

%By similar method, we have the following result.
%\begin{lemma}\cite{HW}
%Let $\Omega$ be  as in \eqref{equ:Hartogs domain}. Suppose that $\Omega$ is
%a simply connected circular domain with center zero and
%the function $-\log\varphi$ is the special K\"ahler potential function determined by the diastatic function of  $g^{D}$.
%Then $(\Omega, g^{\Omega})$ admits a full K\"ahler immersion into $(\mathbb{C}^{\infty}, g_{0})$ if and only
%if $(D, g^{D})$ admits a K\"ahler immersion into $(\mathbb{C}^{n}, g_{0})$ or $(\mathbb{C}^{\infty}, g_{0})$.
%\end{lemma}

By the same method, we can obtain two results below.
\begin{theorem}\label{phd c1}
Let $\Omega$ be  as in \eqref{equ:Hartogs domain}. Suppose that $\Omega$ is
a simply connected circular domain with center zero and
the polarization of  $-\log\varphi$ is the diastatic function of  $g^{D}$.
If $(D, g^{D})$ admits a K\"ahler immersion into $(\mathbb{C}^{n}, g_{0})$ or $(\mathbb{C}^{\infty}, g_{0})$ and
$(D, g^{D})$ is not relative to $(\mathbb{C}^{n}, g_{0})$.
Then $(\Omega, g^{\Omega})$ is not relative to $(\mathbb{C}^{n}, g_{0})$.
\end{theorem}

 \begin{theorem}\label{phd ch1}
Let $\Omega$ be  as in \eqref{equ:Hartogs domain} and  $h$ be a positive number. Suppose that $\Omega$ is
a simply connected circular domain with center zero and
the polarization of  $-\log\varphi$ is the diastatic function of  $g^{D}$.
If  $(D, hg^{D})$ is a K\"ahler submanifold of $(\mathbb{CH}^{\infty}, g_{hyp})$ for $0<h\leq1$ and 
$(D, hg^{D})$ is  not relative to $\mathbb{CH}^{n}$,
 then  it is  not relative to $(\mathbb{CH}^{n}, g_{hyp})$.
 \end{theorem}

\bibliographystyle{amsplain}
%\bibliography{wasserstein-reference.bib}

\providecommand{\bysame}{\leavevmode\hbox to3em{\hrulefill}\thinspace}
\providecommand{\MR}{\relax\ifhmode\unskip\space\fi MR }
% \MRhref is called by the amsart/book/proc definition of \MR.
\providecommand{\MRhref}[2]{
\href{http://www.ams.org/mathscinet-getitem?mr=#1}{#2}}
\providecommand{\href}[2]{#2}

\noindent {Xiaoliang Cheng:\quad
Department of Mathematics, Jilin Normal University, Siping \rm{136000},  PR China
                           Email:  chengxiaoliang92@163.com}\\

\noindent {Yihong Hao:\quad
              Department of Mathematics, Northwest University, Xi'an \rm{710127}, PR China
              Email: haoyihong@126.com}\\

\end{document}